\documentclass[11pt,reqno]{amsart}
\usepackage[english]{babel}
\usepackage{amsmath,amssymb,amsfonts,amsthm}
\usepackage{amscd} 
\usepackage{mathtools} 
\usepackage{mathtext}
\usepackage{longtable}
\usepackage[noadjust]{cite} 
\usepackage{enumerate}
\usepackage{float}
\usepackage[mathcal]{euscript} 
\usepackage[unicode]{hyperref}
\usepackage{ragged2e}
\usepackage{xfrac} 
\usepackage{graphicx}
\graphicspath{{figures/}}
\usepackage[labelfont=small,labelformat=simple]{subcaption}

\usepackage{xcolor}

\theoremstyle{plain}
\newtheorem{theorem}{Theorem}[section]
\newtheorem{lemma}[theorem]{Lemma}
\newtheorem{corollary}[theorem]{Corollary}
\newtheorem{proposition}[theorem]{Proposition}
\theoremstyle{definition}

\newtheorem{example}[theorem]{Example}
\newtheorem{remark}[theorem]{Remark}
\newtheorem{Assumptions}[theorem]{Assumptions}

\usepackage{tikz-cd}
\numberwithin{equation}{section}
\numberwithin{figure}{section}
\begin{document}

\title[Topology change of levels sets in Morse theory]
{Topology change of level sets in\protect\\ Morse theory}

\author{Andreas Knauf and Nikolay Martynchuk}
\thanks{
Department of Mathematics,
Friedrich-Alexander-University Erlangen-N\"urnberg,
Cauerstr.\ 11, D-91058 Erlangen, 
Germany, \texttt{\{knauf, martynchuk\}@math.fau.de}}
  
\begin{abstract}
Classical Morse theory proceeds by considering sublevel
sets $f^{-1}(-\infty, a]$ of a Morse function $f \colon M \to \mathbb R$, where $M$ is a smooth finite-dimensional manifold. In this paper, 
we study the topology of the level sets $f^{-1}(a)$ and give conditions under which 
the topology of $f^{-1}(a)$ changes when passing a critical value. 
We show
that for a general class of functions, which includes all exhaustive Morse function, the topology of
a regular level $f^{-1}(a)$ always changes when passing a single critical point, unless the index of the critical point
is half the dimension of the manifold $M$. When $f$ is a natural Hamiltonian on a cotangent bundle, we obtain more precise results in terms of 
the topology of the configuration space. (Counter-)examples and applications to celestial mechanics are also discussed.
\end{abstract}



\maketitle

\section{Introduction and notation}

Let $M$ be a smooth $m$-dimensional manifold without boundary (in this paper, 
we consider only separable and metrizable manifolds). We recall that a 
function $f\in C^2(M,{\mathbb R})$ is called a Morse function if for every critical point 
$x \in M$ of $f$, the Hessian
$${\rm Hess}f(x):T_{x}M\times T_{x}M\to {\mathbb R}$$
is non-degenerate. One defines the index of a critical point by ${\rm index}(f,x):=\dim(V)$, where $V\subseteq T_{x}M$
is a subspace of maximal dimension on which ${\rm Hess}f(x)$ is negative definite.

The classical Morse theory (see \cite{Mi, Ni} for background material) proceeds by considering {\em sublevel sets}
\begin{equation}
M^b:=\{x\in M\mid f(x)\le c\} \ \ \mbox{ and } \ \  M_a^b:=\{x\in M\mid a\le f(x)\le b\}.
\label{sublevel}
\end{equation}
Under standard compactness assumptions (for instance, if all the sets $M^b_a$ are compact; see Section~\ref{sec:level_sets}), for regular values $a < c$ of $f$, 
the sublevel sets $M^a$ and $M^b$ are diffeomorphic if $M_a^b$ contains no critical points. 
If there is one critical
point $x$ in the interior of $M_a^b$, then $M^b \cong M^a \cup H^m_k$, with the handle 
$H^m_k:=D^k\times D^{m-k}$ of index $k:= {\rm index}(f,x)$ attached. In this case, $M^b$ is not homotopy equivalent to $M^a$ if $M$ is compact, see \eqref{PPPQ} below.

In the present paper we are interested in the topology of the {\em level sets} $f^{-1}(a)= \partial M^a$. First we note that
$\partial M^b$ and $\partial M^a$ are always diffeomorphic if $M_a^b$ contains no critical points. If
$M_a^b$ contains some critical points, then both scenarios are possible. 
Simple examples (like Ex.\ \ref{ex:nonor} below) show that for regular values $a<b$,  
$\partial M^b$ may be diffeomorphic to $\partial M^a$ even when $M_a^b$ contains a single critical point.
A natural question is to understand when the {\em topology does change} 
(in the rough sense that $H_\ell(\partial M^b,G)\neq H_\ell(\partial M^a,G)$ 
for some abelian group $G$ and some $\ell\in{\mathbb N}_0$)
when the function $f$ passes a critical level.

We will develop criteria to answer this question in many cases.
More specifically, in Section \ref{sec:level_sets} we consider level sets of abstract Morse functions, which satisfy 
the so-called Palais-Smale condition \cite{PS} and for which the level sets have finitely generated homology groups; see Assumptions~\ref{Assumptions}.  We show that 
for such a function $f$, the topology of
$f^{-1}(a)$ changes when passing a single critical point, if
the index $k$ of the critical point
is different from $m/2$, where $m$ is the dimension of the manifold $M$. We also consider the case of
several critical points on a given critical level and prove the topology change under a certain assumption
on the indices of these critical points. Specifically, it turns out that the topology always changes when passing a given critical level if it contains  a 
critical point of index $k \ne m/2$ 
such that there exists no other critical point 
of index $k-1, k+1$ or $m-k$; see Theorem~\ref{theorem:level2}.

The results obtained in Section~\ref{sec:level_sets} are quite general; they apply to functions which do not have to be everywhere Morse (it is sufficient to assume non-degeneracy  
in a small neighborhood of a given critical level), and the level sets do not have to be compact. The crucial assumption that is necessary
is the existence of a critical point whose index is not half the dimension of the underlying manifold. If the index is in the middle dimension, then, at least on the level of Morse function on abstract 
manifolds, both outcomes are possible.
Nonetheless, it is natural to ask for criteria which guarantee the topology change also in this situation. In  Section \ref{sec:natural}, we consider the case when the manifold
$M$ is a rank $n$ vector bundle over an $n$-manifold $N$, and the function $f$ is a fiberwise positive definite quadratic form. 
The most important such case is the one of a Hamiltonian function 
\begin{equation}
H:M = T^*N\to {\mathbb R}\quad,\quad H(q,p)= K(p-A(q))+V(q)
\label{Hamilton}
\end{equation} 
on the cotangent bundle of a configuration manifold $N$; here $K$ and $V$ are the system's kinetic and potential energy, and $A$ is the magnetic potential. In this case, we show that the topology of $H^{-1}(h)$ always changes when passing 
a single critical point if the Euler characteristic of the configuration space $N$ is different from $\pm 1$. In the case of 
abstract vector bundles, we obtain a similar result in terms of the Euler number. We note that the Euler number
plays an important role also in the context of classification of integrable Hamiltonian systems with two degrees of freedom 
\cite{FZ, BF} and monodromy of such systems \cite{EM, ME, MBE}. In Section \ref{sec:natural} we also consider
the special case when $N$ is an $n$-sphere. Because of Adams' result on the Hopf invariant one problem \cite{Ada}, it follows that in this case, no topology change is possible only in
dimensions $n = 2, 4$ and $8$.

The problem that we address in this paper was raised by A. Albouy, in connection with the $n$-body problem and the topology of the corresponding integral 
manifolds; see \cite{Sm1, Sm2, Ala}.  The main question is
whether for this problem, the topology of the integral manifolds always changes when passing through a bifurcation level.
This is obvious for the two levels of the Kepler problem with a given nonzero value of angular momentum. 
As shown in \cite{MMW}, this is the case also for $n=3$ celestial bodies. In Section~\ref{section/nbody_problem}, we 
show that this is also
true for the planar $n$-body problem, provided that the reduced Hamiltonian is a Morse function having at most two critical points
on each level set; cf.  \cite{McC}. Two other examples from classical mechanics, which illustrate the theory, are also discussed.

The paper is concluded with the Appendix,
where details of some proofs and a few miscellaneous results are given.

\section{Level sets of Morse functions} \label{sec:level_sets}

Let $f\in C^2(M,{\mathbb R})$ be a Morse function on a manifold $M$ and let $m := \dim(M) > 0$ denote the dimension of the manifold. 
It is not difficult to see that two level sets $f^{-1}(a)$ and $f^{-1}(b)$ may be diffeomorphic even if there are critical values
in the subinterval $[a,b]$. Indeed, one can take a closed manifold $M$ and a Morse function $f$ on this manifold
that has unique global minimum and maximum points $x_{\min}$ and $x_{\max}$; then 
$f^{-1}(f(x_{\max})-\varepsilon)$
and $f^{-1}(f(x_{\min})+\varepsilon)$ are diffeomorphic to a sphere $S^{m-1}$. 

This is also possible if between the levels there is only one critical level with 
a single critical point:
\begin{example}[No topology change of level sets]\label{ex:nonor}\quad\\ 
Consider the perfect Morse function $f=\tilde{f}\circ \pi^{-1}:{\mathbb R}{\rm P}(2)\to {\mathbb R}$, 
induced by
\[
\tilde{f} : S^2\subseteq {\mathbb R}^3\to {\mathbb R}\quad,\quad 
\textstyle \tilde{f}(x)=\sum_{k=1}^3 k\,|x_k|^2, \]
with projection $\pi \colon  S^2\to {\mathbb R}{\rm P}(2)\cong S^2/S^0$.  
Then 2 is a critical value of $f$ with the unique critical point $\pm(0\ 1\ 0)$, and
$f^{-1}(2+\varepsilon) \cong S^1 \cong f^{-1}(2-\varepsilon)$.
Note that a similar phenomenon arises for a perfect Morse function $f$ on the complex projective space ${\mathbb C}{\rm P}(2) \cong S^5/S^1$, 
with level sets diffeomorphic to the sphere $S^3$. So it is not caused by lack of orientability.

More generally there is no topology change  for $m$ even at the level of the critical point with 
index $m/2$ of a perfect Morse function on ${\mathbb R}{\rm P}(m)$, respectively index $m$ for 
${\mathbb C}{\rm P}(m)$.
\hfill$\Diamond$
\end{example}
So we consider the following general question:
{\it 
How does the homology of level sets of a Morse function $f:M\to {\mathbb R}$ 
on an $m$-dimensional manifold $M$ change when passing a critical value $c\in {\mathbb R}$?}

First, consider the more usual case of sublevel sets \eqref{sublevel} and assume, for the moment, that 
$f$ is {\em exhaustive}, that is, that all the sublevel sets 
$$M^a = f^{-1}(-\infty,a]$$ are compact.
Consider $a<c<b$ such that there is only one critical point $x_c\in f^{-1}(c)$ of $f$
in $M_a^b$.  
For $k:={\rm index}(f,x_c)$, 
\begin{enumerate}[1.]
\item 
$M^b$ is {\em homotopy equivalent} to $M^a$ with a $k$-cell $D^k$ attached.
\item 
$M^b$ is {\em diffeomorphic} to $M^a$ with an $m$-dimensional handle 
$H^m_k:=D^k\times D^{m-k}$ of index $k$ attached, see \cite[Thm.\ 17.5]{DFN}.
Note that this uses as data the embedding $S\to \partial M^a$
of a sphere $S:=S^{k-1}$ with trivial normal bundle $T_S^\perp (\partial M^a)$ 
and, see \cite[page 24]{Ni}, a bundle isomorphism 
\[\varphi: S\times {\mathbb R}^{m-k} \to T_S^\perp (\partial M^a).\]
\end{enumerate}
Then by excision \cite[Sect.\ 2.3]{Ni}, one has
\[H_\bullet(M^b, M^a)\cong H_\bullet(D^k,\partial D^k).\]
Here $H_\bullet(X, A) \equiv H_\bullet(X,A, {\mathbb F})$ denotes the relative homology chain
complex of a pair $(X,A)$. (Unless stated otherwise, the coefficients are in a field $\mathbb F$.)
By subadditivity  \cite[Lemma 2.14]{Ni} 
for the long exact homological sequence of the pair $(M^b, M^a)$,
one gets the inequality
\[P(M^a)+P(M^b,M^a)\succeq P(M^b)\]
of Poincar\'e polynomials, with relative Poincar\'e polynomial $P(M^b,M^a)(t)=t^k$. 
The symbol $\succeq$ means existence of a polynomial $Q$ with nonnegative coefficients and
\begin{equation}
 \big(P(M^a) + P(M^b,M^a) - P(M^b)\big)(t) = (1+t)Q(t).
 \label{PPPQ}
 \end{equation}
This implies, in particular, that $P(M^a)\neq P(M^b)$. 

Hence, if $f$ is exhaustive (in particular, if $M$ is compact) and $M_a^b$ contains a critical level $f^{-1}(c)$ with one 
critical point, then $M^a$ is not homotopy equivalent to $M^b$. As we have seen above, for the level sets $\partial M^a=f^{-1}(a)$ this is no longer the case.

We note that if $M$ is non-compact, it may happen that the level and sublevel sets change their topology even in the absence of critical points.
This phenomenon occurs, for instance, in the $3$-body problem and in that case is due to the so-called critical points at infinity \cite{Ala}. To avoid this situation, but 
to include a large class of functions on non-compact manifolds,
we shall assume throughout this paper that  $f$ satisfies the following {\bf assumptions}; cf. \cite{PS}.

\begin{Assumptions} \label{Assumptions}
1.  There exists a Riemannian metric on $M$ such that $M$ is complete with respect to this metric and 
 for every $S \subset M$ on which $|f|$ is bounded, but the norm  $\|\textup{grad} f\|$ is not bounded away from zero, there is a critical point of $f$ in the closure of $S$.\\[1mm]
 2. The integer homology groups of each level set $\partial M^b = f^{-1}(b)$ are finitely generated.
\end{Assumptions}

\begin{remark}
The first assumption is known as the  Palais-Smale condition \cite{PS}. Under this condition,  the usual Morse theory applies, even if the level sets 
$f^{-1}(a)$ are not compact. The second condition is technical and is needed for dimension counting. It 
implies, for instance, that homology groups with coefficients in a field $\mathbb F$ are finite-dimensional; 
see Proposition~\ref{proposition/hom_change}.
We note that when $f$
is a {\em proper} function, that is,
 when all the sets $M^b_a = f^{-1}[a,b]$ are compact (in particular, when $f$ is exhaustive),
 Assumptions~\ref{Assumptions} are satisfied. Non-compact examples
 can be found in Sections~\ref{sec:natural} and \ref{section/nbody_problem}.
\end{remark}

We will need the following proposition, where the above notation is understood.

\begin{proposition} \label{proposition/hom_change}
Let $x_c$ be a non-degenerate critical point of $f$ with $k = {\rm index}(f,x_c) < m/2$ and let $\mathbb F$ be a field. 
Assume that values $a$ and $b$ are chosen such that $x_c$ is the only critical point in $M^b_a.$
Then we have (under Assumptions~\ref{Assumptions})
 $$\dim H_{\ell}(\partial M^b, \mathbb F) = \dim H_{\ell}(\partial M^a, \mathbb F) +j_\ell,$$ 
where $j_\ell$ is an integer such that

$1)$ $j_\ell = 0$ if $\ell \notin \{k-1, k, m-k-1, m-k, m-2, m-1\}$;

$2$) $j_{k-1}j_k = 0$ and, moreover,
\begin{itemize}

\item[$\ 2.a$)] $j_{k-1} = 0$ implies $j_{k} \in \{1,2\}$; if $2k \ne m-1$, then $j_{k} = 1;$

\item[$\ 2.b$)] $j_{k} = 0$ implies $j_{k-1} = -1$. 
\end{itemize}
\end{proposition}
\begin{example}\quad\\[-5mm]
\begin{itemize}
\item  [$\ 2.a$)]
$j_1 = 2$,  $j_0=0$, $2k = m-1$: 
Attaching a handle $S^1\times D^1$ to $S^2\subseteq \mathbb{T}^3$ yields 
$\mathbb{T}^2\subseteq \mathbb{T}^3$.
\item [$\ 2.a$)]
$j_1 = 1$,  $j_0=0$, $2k \ne m-1$: The level sets of a four-manifold
may bifurcate from $S^3$ to $S^2\times S^1$.
\item [$\ 2.b$)]
$j_1 = 0$,  $j_{0}=-1$: 
Connecting $S^2\sqcup S^2$ by a handle $S^1\times D^1$ yields $S^2$.
\end{itemize}
\end{example}
\begin{proof}
The boundary of the handle $H^m_k$ has the form
\[\partial H^m_k = (\partial D^k\times D^{m-k})\cup (D^k\times \partial D^{m-k}) =
(S^{k-1}\times D^{m-k}) \cup (D^k\times S^{m-k-1}). \]
By the above, $\partial M^b$ is diffeomorphic to $\partial (M^a \cup H^m_k)$, that is,
\[\partial M^b \cong \big(\partial M^a\setminus (S^{k-1}\times D^{m-k}) \big) \cup_\varphi   (D^k\times S^{m-k-1}).\]
With 
\begin{equation} \label{eq:U:V}
U:=\partial M^a\setminus (S^{k-1}\times D^{m-k})\mbox{ and }V := D^k\times S^{m-k-1}
\end{equation}
this is abbreviated as $\partial M^b \cong U\cup_\varphi V$.
Observe that
\begin{align*}
H_\bullet(\partial M^a,U)& \cong H_\bullet(S^{k-1}\times D^{m-k}, S^{k-1}\times S^{m-k-1}) \cong  
{\mathbb F}^{\delta_{m-k}(\bullet)}\oplus {\mathbb F}^{\delta_{m-1}(\bullet)}
\end{align*}

This can be proven using cellular homology and a CW decomposition of $S^{k-1}\times D^{m-k}$ into the union of two cells of dimensions 
$m-k$ and $m -1$. Alternatively, one can use the relative K\"unneth formula in homology; see Dold \cite[Chap.\ VI.10]{Do}.  
Similarly, one has
\begin{align*}
H_\bullet(\partial M^b,U) &
\cong H_\bullet(D^k\times S^{m-k-1},S^{k-1}\times S^{m-k-1}) 
\cong {\mathbb F}^{\delta_{k}(\bullet)}\oplus {\mathbb F}^{\delta_{m-1}(\bullet)}
\end{align*}

By the long exact homology sequence of the pair $(\partial M^a,U)$
\[ \ldots \stackrel {{\partial}_{\ell+1}^a}{\to} H_\ell(U)\stackrel {i_\ell^a}{\to}  H_\ell(\partial M^a) 
\stackrel {j_\ell^{a}}{\to}  
H_\ell(\partial M^a,U)\stackrel {{\partial}_\ell^a}{\to} H_{\ell-1}(U)\stackrel {i_{\ell-1}^a}{\to}   \ldots,\]
and similarly for $(\partial M^b,U)$,
changes in homology between $M^a$ and $M^b$ can at most happen in dimensions 
$\ell \in \{k-1, k, m-k-1, m-k, m-2, m-1\}$:

\begin{tikzcd}[row sep=tiny]\hspace*{-5mm}
 {\mathbb F}^{\delta_{m-k-1}(\ell)}\oplus {\mathbb F}^{\delta_{m-2}(\ell)} \hspace*{-4mm}\arrow{dr}{{\partial}_{\ell+1}^a}  
 &&H_\ell(\partial M^a)\arrow{r} {j_\ell^a} &  {\mathbb F}^{\delta_{m-k}(\ell)}\oplus {\mathbb F}^{\delta_{m-1}(\ell)}\\
 \hspace*{-5mm}
&H_\ell(U) \arrow{ur}{i_\ell^a} \arrow{dr}{i_\ell^b} &\\
\hspace*{-5mm}
 {\mathbb F}^{\delta_{k-1}(\ell)}\oplus {\mathbb F}^{\delta_{m-2}(\ell)} \arrow{ur}{{\partial}_{\ell+1}^b}
 && H_\ell(\partial M^b)\arrow{r} {j_\ell^b} & {\mathbb F}^{\delta_{k}(\ell)}\oplus {\mathbb F}^{\delta_{m-1}(\ell)}
\end{tikzcd}

Set
$$a_\ell:=\dim(H_\ell(\partial M^a))\ ,\ b_\ell:=\dim(H_\ell(U))\ ,\  c_\ell:=\dim(H_\ell(\partial M^b)).$$

We shall also assume that $k > 0$, as the theorem is trivial for $k = 0$.

Under the above assumptions on $k$, we have that 
$$\delta_{m-1}(k-1) = \delta_{k}(k-1) = \delta_{m-k}(k-1) = \delta_{m-k-1}(k-1) 
= \delta_{m-2}(k-1) = 0.$$
It follows that  $i_{k-1}^a$ is an isomorphism and  that $i_{k-1}^b$ is surjective 
(but not necessarily injective).

Consider the homology cycle ${\partial}_{k}^b[D^k]$, which is an element of $H_{k-1}(U)$. 
There are two possibilities, depending on whether this cycle is zero in $H_{k-1}(U)$. 
If it is non-zero,  then the homomorphism $i_{k-1}^b$ has a non-trivial kernel and hence $a_{k-1} - c_{k-1} = b_{k-1} - c_{k-1} > 0.$ Since 
${\partial}_{k}^b[D^k]$ spans the kernel of $i_{k-1}^b$, we have that $c_{k-1} = a_{k-1} -1.$

Now suppose ${\partial}_{k}^b[D^k] = 0.$ 
We shall show that in this case $c_k = a_k +1$ or $c_k = a_k + 2$. 
Indeed, we observe that ${\partial}_{k}^b[D^k] = 0$ implies the existence of $\alpha_k$ in 
$H_k(\partial M^b)$ such that $j_k^b(\alpha_k) = [D^k].$ 
In particular, $\alpha_k$ is not in the kernel of $j_k^b$ and, by exactness, not in the image of $i_{k}^b$.
First, assume that $k \ne m-2$. Then $i_{k}^b$ is injective and hence $b_k  =  c_k -1 $. 
On the other hand, $i_{k}^a$ is surjective and hence
$a_k = b_k$ or $a_k = b_k - 1$. 
We infer that $c_k - a_k = 1$ or $2$. Note that if $2k \ne m-1$, then $c_k = a_k + 1$. 
Now assume that $k = m-2$. Then $m \le 3$ by the assumption $2k < m$. 
\end{proof}

In the following proposition we use $G = \mathbb Z$ as a group of coefficients for homology.
With notation \eqref{eq:U:V} we obtain:

\begin{proposition} \label{proposition/middle_dimension} 
Consider the case of the middle dimension $2k = m$ and a single critical point.
The levels $\partial M^b$ and $\partial M^a$ have different homology groups
if the orders of the group elements ${\partial}_{k}^a[D^{m-k}]$ and ${\partial}_{k}^b[D^{k}]$ 
are different in $H_{k-1}(U,\mathbb Z)$. 
This is the case, in particular, if exactly one of the classes ${\partial}_{k}^a[D^{m-k}]$ and ${\partial}_{k}^b [D^{k}]$ vanishes.
\end{proposition}

Observe that the function $g = -f$ has the same level sets as the function $f$; 
moreover, each index-$k$ critical point of $f$ is also a critical point of $g$ of index $m - k$. 
From Proposition~\ref{proposition/hom_change}, we get the following result. 
\begin{theorem} \label{theorem:level}
Let $f$ be a Morse function on an $m$-manifold $M$ with sublevel sets $M^b = f^{-1}(-\infty, b]$ and $M^b_a = f^{-1}[a, b]$. Assume that 
the set $M^b_a$ contains $N \ge 1$ critical points of the same index $k$ and no critical points of some other index. If the homotopy type of $\partial M^b$ and of $\partial M^a$ coincide, then 
$k=m/2$ (so in particular $m$ is even). 
\end{theorem} 

Slightly more generally, we have the following theorem.

\begin{theorem} \label{theorem:level2}
Let $f$ be a Morse function on an $m$-manifold $M$. Assume that 
the set $M^b_a$ contains a critical point of index $k \ne m/2$ 
such that there exists no other critical point in $M^b_a$
of index $k-1, k+1$ or $m-k$. 
Then the homotopy types of $\partial M^b$ and of $\partial M^a$ do not coincide. 

\end{theorem} 
\begin{proof}
First observe that the statement holds if $m \le 3$. 
We shall assume that $m > 3$ and that $k < m/2$; the case $k > m/2$ is similar. 

The index-$k$ critical point contributes to the change of the Betti numbers $b_{k-1}$ or $b_{k}$, 
when the value exceeds $f(x_k)$. 
Specifically, when passing this critical point $b_{k-1}$ decreases by $1$ or $b_k$ 
increases by $1$ or $2$; see Proposition~\ref{proposition/hom_change}. 
Any critical point of index $k'$ may 
contribute to the homology change only in the following dimensions: 
$\{k'-1, k', m-k'-1, m-k', m-2, m-1\}$. Assuming $k < m/2$
$m > 3,$ and $k' \ne k+1, k-1, m-k$, we shall now show what possible changes 
in the Betti numbers $b_{k-1}$ and $b_k$ can occur when passing such a critical point.
 
Observe that  $m-1, m-2 \notin \{ k, k-1 \}$. 
Moreover, the index $k'$ critical point gives rise to changes in $b_{k-1}$ and $b_k$ only if
$k' \in \{k, m-k+1,m-k-1\}.$ 
 
\begin{enumerate}[$\bullet$]
\item 
If $k' = k$, then $b_{k-1}$ decreases by $1$ or $b_k$ increases by $1$ or $2$. 

\item 
If $k' = m-k + 1$, then $m-k' = k-1$. 
Hence when passing the critical point of index $k' = m-k+1$, the Betti number $b_{k-1}$ 
can only decrease and  the Betti number $b_k$ does not change.
 
\item 
Finally, consider the case $k' = m-k-1$. If $k' < m/2,$ then $k' = k$. 
This case was considered earlier. 
If $k'>m/2$, then $m-k' = k+1 < m/2$. 
Moreover, the Betti number $b_k$ can only increase and $b_{k-1}$ does not change. 
The remaining case $k' = m/2$ is not possible since then $k+1 = m/2 = k'$.
\end{enumerate}
  
We observe that in all these cases the number $b_{k-1}$ can only decrease and $b_k$ can only increase. Thus, the initial change
that occurs when passing the given index $k$ critical point cannot be compensated. 
The result follows. \end{proof}

\section{Energy levels in classical mechanics} \label{sec:natural}

\subsection{Mechanical systems on vector bundles} \quad\\
Consider a rank $n$ vector bundle 
\begin{equation} \label{def:pi:E:M}
\pi \colon E \to N
\end{equation}
over a connected $n$-manifold $N$ without boundary. The manifold $N$ and the bundle
$\pi \colon E \to N$ are assumed to be orientable. 

We will be interested in the topology change of level sets of
a `Hamiltonian' function on $E$ of the following form \eqref{Hamilton} (which includes the class of natural mechanical systems
and natural mechanical systems with magnetic terms)
\begin{equation}
H = K + V \circ \pi,
\end{equation}
where $K$ is a Riemannian bundle metric on $E$ and $V$ is a Morse function on $N$.  We shall assume that both
$H$ and $V$ satisfy Assumptions~\ref{Assumptions}. We will need 
the following result, which specifies how the homology groups of the level sets $H^{-1}(h)$ change when passing an index-$n$ critical point.

\begin{proposition} \label{proposition/maxima}
Let $x_c$ be a non-degenerate local maximum of $\,V$ such that there are no other critical points on $V^{-1}(h_c), \ h_c = V(x_c) = H(x_c,0).$ Let $\mathbb F$ be a coefficient field and  $\varepsilon > 0$ be sufficiently small. Then
the $(n-1)$-Betti number changes according to
$$
b_{n-1}\big(H^{-1}(h_c+\varepsilon),\mathbb F\big) 
= b_{n-1}\big(H^{-1}(h_c-\varepsilon),\mathbb F\big) + j_{n-1},
$$
where $j_{n-1} = -1$ if $x_c$ is not a global maximum and $j_{n-1} \in \{-1,0,1\}$ if $x_c$ is a global maximum.
\end{proposition}

\begin{proof}

Let $x_c$ be a local maximum of $V$ and $h_c$ be the corresponding critical value. By Assumptions~\ref{Assumptions},
for all $\varepsilon > 0$ small, $h_c$ is the only critical value of $V$ 
(and therefore also of $H$) in $[h_c-\varepsilon,h_c+\varepsilon]$. 
We observe that: 

\begin{enumerate}[1.]
\item 
By applying the Morse Lemma \cite[Thm.\!~1.12]{Ni} to $V$,
there is  a closed neighborhood $U\subseteq N$ of 
$x_c$ and a suitable chart $\psi:U\to \psi(U)=D^m\subseteq {\mathbb R}^m$ 
with $\psi(x_c)=0$, such that $V\circ \psi^{-1}(x)=h_c-\|x\|^2$.

\item
Following a proof of the Morse lemma (for the function $H$), we can find a closed neighborhood $\tilde{U}$ of $\tilde{x}_c:=(x_c,0)$ 
and a local trivialization $\varphi \colon \tilde{U} \to D^n\times  D^n$
of \eqref{def:pi:E:M} such that $\psi \circ\pi = \pi_1\circ \varphi$ for $\pi_1(x,y):=x$ and
\begin{equation}
\label{H:Morse}
H\circ \varphi^{-1} \, (x,y) = h_c + \|y\|^2 - \|x\|^2 \qquad \big((x,y)\in D^n\times  {D}^n\big).
\end{equation}
If $\varepsilon>0$ is small, then the intersections with $W = \pi^{-1}(U)$ 
of the levels 
$$
\Sigma_\pm := H^{-1}(h_c\pm \varepsilon)
$$
are contained in $\tilde{U}$; moreover, they are the images of diffeomorphisms 
$$ \hat{f}_-: S^{n-1} \times D^n\to \Sigma_-\cap W, \ \mbox{  respectively, }\
\hat{f}_+: D^n \times S^{n-1} \to \Sigma_+ \cap W;
$$
see Figure~\ref{fig:morseLemma}. 
\begin{figure}[htbp]
\begin{center}
\includegraphics[width=70mm, height = 61mm]{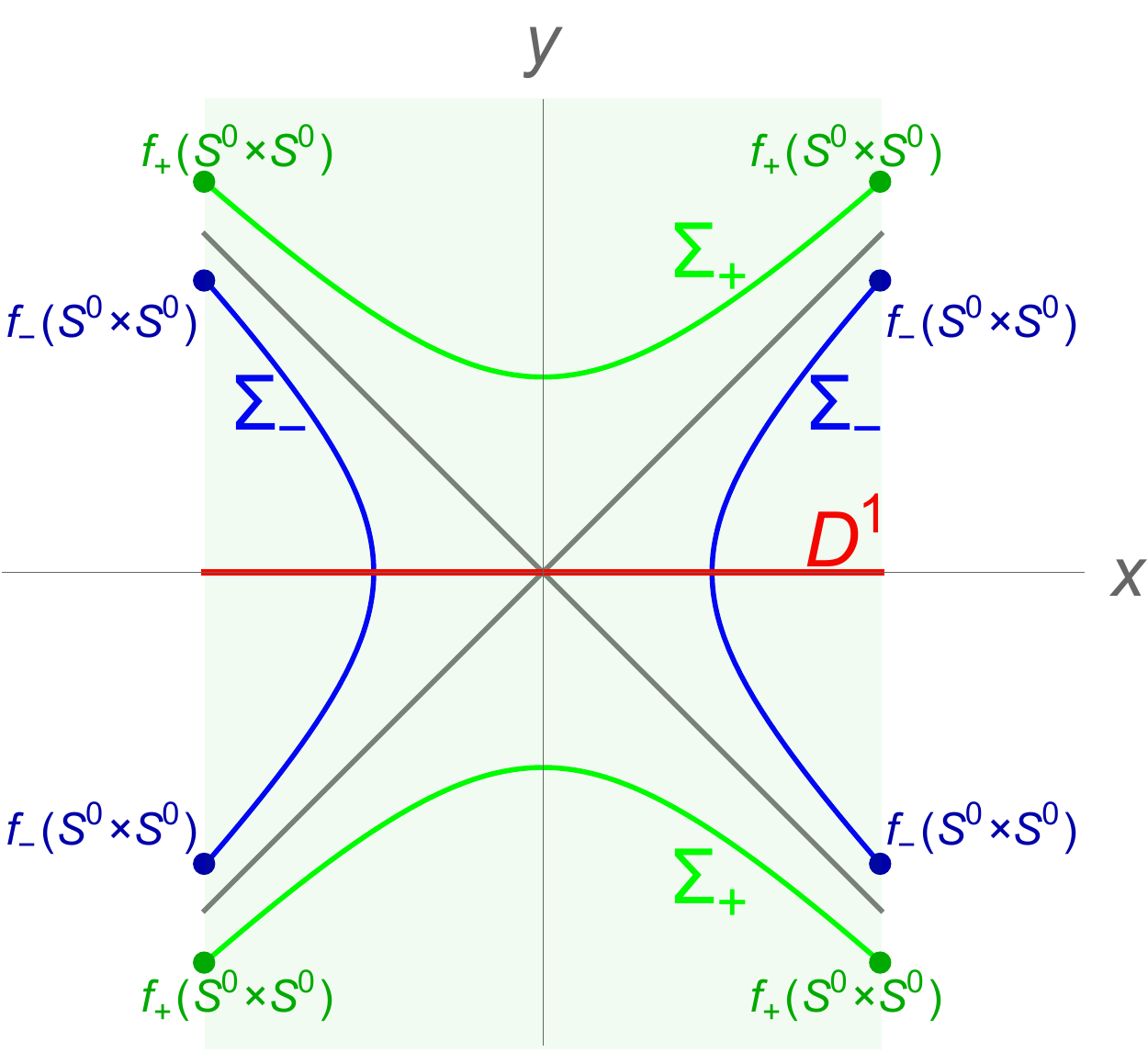}
\caption{The neighborhood $\tilde{U} \subset \pi^{-1}(U)$ of the critical point $\tilde{x}_c$
for $n=\dim(N)=1$.}
\label{fig:morseLemma}
\end{center}
\end{figure}
\item 
The  
$(2n-1)$--manifolds with boundary 
$$ T_\pm := \Sigma_\pm \setminus  \textup{int}(W)
 $$
are naturally homeomorphic.

\end{enumerate}

The third statement follows, for instance, from the following construction. Specifically, consider
a regular level set $H^{-1}(h)$. It projects to the corresponding {\em Hill region} 
$$B_h := \{x \in N \mid V(x) \le h\},$$ 
which is a manifold with boundary. The level set 
$H^{-1}(h)$ can be viewed as a sphere  bundle with fibers $S^{n-1}_{r(x)}$ over $x\in B_h$ 
of radius $r(x) = \sqrt{2(h-V(x))}$, thus collapsed over $\partial B_h$.

By 1) and 2), using the chart $(\tilde{U},\varphi)$, we can write
$$
\Sigma_- \cong  T_- \cup_{f_-} S^{n-1} \times D^n\ \mbox{ and }\ 
\Sigma_+ \cong  T_+ \cup_{f_+} D^n \times S^{n-1}
$$
where the attaching maps for the respective boundary components 
are the diffeomorphisms 
\[f_\pm:=\hat{f}_\pm|_{S^{n-1} \times S^{n-1} }:S^{n-1} \times S^{n-1} \to\partial T_\pm 
\quad , \quad \textstyle
(x,y)\mapsto  \big(x,\sqrt{1\pm \varepsilon}\, y\big) .\]
From this description it follows that under the map $f_{-}$, 
the boundary of the $n$-disk $\{x\} \times D^n$ is mapped to a fiber of the collapsed sphere bundle
$\pi \colon T_{-} \to \pi(T_{-})$. On the other hand, $f_{+}(\partial D^n \times \{y\})$ is a cross section. We note that 

$[\{x\}\times D^n]$ generates $H_n(S^{n-1}\times D^n,S^{n-1}\times S^{n-1}, \mathbb F)$,
if $x\in S^{n-1}$;

$[D^n\times \{y\}]$ generates
$H_n(D^n\times S^{n-1},S^{n-1}\times S^{n-1}, \mathbb F)$, if
$y\in S^{n-1}$.\\
Consider the homology exact sequences of the pairs
$(\Sigma_\pm,T_\pm)$:
\begin{equation} \label{eq/exactseq1}
\mathbb F \stackrel{\partial_{*}}{ \longrightarrow} H_{n-1}(T_\pm, {\mathbb F}) \longrightarrow
H_{n-1} \big(\Sigma_\pm, {\mathbb F}\big)  \longrightarrow 0.
\end{equation}
Below we shall use this sequence, together with the information about the attaching maps $f_{\pm}$ and some additional properties to
compare  the $(n-1)$-th homology groups of $\Sigma_+$ and $\Sigma_-$.

First, consider the case of a global maximum. 
In this case, the statement follows from the exactness of Eq.~\eqref{eq/exactseq1}.

Now consider the case when $x_c$ is a local, but not a global maximum. We observe that the following properties hold:
\begin{itemize}
\item[A)]
Let $B_\pm:= B_{h_c\pm\varepsilon}$. 
The intersection $\partial \pi(T_+)\cap \partial B_{+} \neq\emptyset$.
\item[B)] 
For any $x \in  \pi(T_-)$, there is a homotopy within $T_-$ between the fiber 
$S^{n-1}_{r(x)}  \hookrightarrow T_-$ over $x$ 
(with $r= \sqrt{2(h_c-\varepsilon-V(x))}$) and a point. 
\end{itemize}

To prove A), we note that $h_c<h_c+\epsilon < \sup_q V(q)$, so that 
$\partial B_+ \neq \emptyset,$ and that $\pi(T_+) = B_+ \setminus  U$. 

To prove B), we construct a homotopy $p: S^{n-1}\times [0,1]\to T_-$ which projects
to a path $\pi\circ p:[0,1]\to B_-$ from $x_c$ to a point of 
$\partial B_-\setminus \psi^{-1}\big( S^{n-1}_{\sqrt{\varepsilon}}\big)$.

\medskip

We note that properties A) and B) do not hold if $x_c$ is a global maximum. If $x_c$ is only a local maximum, then 
$\dim (H_{n-1}(\Sigma_{+})) = \dim (H_{n-1}(\Sigma_-)) -1,$ as we now show.
\begin{enumerate}[$\bullet$]
\item 
Consider the case of $\Sigma_+$. 
The first of the maps in Eq.~\eqref{eq/exactseq1} is given by a boundary homomorphism
$\partial_{*}$ on $H_n(D^n\times S^{n-1},S^{n-1}\times S^{n-1}, \mathbb F)
\cong \mathbb F$.
We claim that the image of $\partial_{*}$ is non-trivial in this case. 
Indeed, by A), 
$[D^n\times \{y\}] \in H_n(D^n\times S^{n-1},S^{n-1}\times S^{n-1}, \mathbb F)\setminus \{0\}$. 
We observed above that $f_{+}(\partial D^n \times \{y\})$ is a cross section over $\pi(f_{+}(\partial D^n \times \{y\}))$,
where $\pi$ is defined in \eqref{def:pi:E:M}. 
It follows that
$$\pi_{*}(\partial_*[{f}_+(D^n\times \{y\})]) = [\partial U].$$ 
In particular, we have that $\partial_*[{f}_+(D^n\times \{y\})] \in H_{n-1}(T_+)$ is non-zero, using B). 
From the exactness of \eqref{eq/exactseq1} it follows that the map 
$H_{n-1}(T_+) \to H_{n-1}(\Sigma_+)$ is not injective. 
However, it is surjective by virtue of the last arrow in \eqref{eq/exactseq1}.
\item 
Consider the remaining case of $\Sigma_-$. We observe that the map 
\[H_{n-1}(T_-) \to H_{n-1}(\Sigma_-)\] 
is bijective; indeed, from C) it follows that $\textup{image}(\partial_{*}) = 0$ in this case.
\end{enumerate}
We recall that,  by the assumption, 
the homology groups of $T_\pm$ and $\Sigma_\pm$ are finite dimensional ${\mathbb F}$-vector spaces. 
But by Observation 3) above, the spaces $T_-$ and $T_+$ are homeomorphic.
By counting dimensions, we conclude that 
$\dim(H_{n-1}(\Sigma_+)) = \dim(H_{n-1}(\Sigma_-)) - 1.$
\end{proof}

Combining Theorem~\ref{theorem:level} and Proposition~\ref{proposition/maxima}, we get

\begin{corollary} \label{corollary/not_global_maximum}
Consider a function on $E$ of the form
$$
H = K + V \circ \pi,
$$
where $K$ is a Riemannian bundle metric on $E$. 
Let $x_c$ be a non-degenerate critical point of $V$ that is not a global maximum. 
Assume that $x_c$ is the only critical point on $V^{-1}(h_c)$, for $h_c = V(x_c) = H(x_c,0)$.
Then the topology of $H^{-1}(h)$ changes when $h$ passes the critical value $h_{c}$. 
\end{corollary}

\begin{corollary} \label{corollary/many_global_maxima}
  Consider a function on $E$ of the form
$$
H = K + V \circ \pi,
$$
where $K$ is a Riemannian bundle metric on $E$. 
Let $x_1, \ldots, x_L \in V^{-1}(h_c)$ be the non-degenerate global maxima of the function $V$ on $N$.  If $L \ge 3$, then
  the topology of $H^{-1}(h)$  changes when $h$ passes the critical value $h_c$.
\end{corollary}

\begin{remark}
 We note that Proposition~\ref{proposition/maxima} and 
 Corollaries~\ref{corollary/not_global_maximum}, \ref{corollary/many_global_maxima}  hold 
 even when the bundle \eqref{def:pi:E:M} is not orientable (but the base $N$ still is). 
 \end{remark}

Below we shall study in more detail the case of a global maximum of $V$. 
We shall additionally assume that the base manifold $N$ is compact. 
Under this assumption, the Euler number $e(E)$ of $\pi \colon E \to N$ is defined.

The following result specifies when the topology of $H^{-1}(h)$ changes 
when passing a global maximum of $V$ or a number $N$ of global maxima, which are on the same energy level. 

\begin{theorem} \label{theorem/closed_manifold}
Let $N$ be a closed orientable $n$-manifold and $\pi \colon E \to N$ 
be an orientable rank $n$ vector bundle over $N$. 
Consider the function $H = K + V\circ \pi \colon E \to {\mathbb R}$ on $E$
and let $x_1, \ldots, x_L \in V^{-1}(h_c)$ be the non-degenerate global maxima of the function 
$V$ on $N$.  
Then the topology of $H$ changes when $H$ passes the critical value $h_c$ 
if one of the following conditions is satisfied
  
$1)$ $L = 1$ and the Euler number $e(E)$ is not equal to $\pm 1$;
 
$2)$ $L = 2$ and the Euler number $e(E)$ does not vanish;

$3)$ $L > 2.$
\end{theorem}

\begin{proof}
The case $L > 2$ was considered earlier. The proof in the other cases relies on the homology exact sequence of a pair. In the case $L = 1$, a special choice
of a coefficient group is made; specifically --- $\mathbb Z_k$, where $k = e(E)$ is the Euler number. Details are given in Appendix~\ref{section/appendix}; see 
Theorem~\ref{theorem/closed_manifold_appendix}.
 \end{proof}

\begin{corollary} \label{corollary/one_global_max}
Let $N$ be a closed orientable $n$-manifold and $\pi \colon E \to N$ be an orientable rank $n$ vector bundle over $N$. If the Euler number $e(E) \ne \pm 1$, then 
the topology of $H = K + V\circ \pi \colon E \to {\mathbb R}$ changes whenever $H$ passes 
a critical level $H^{-1}(h_c)$ with one (and only one) non-degenerate critical point. 
\end{corollary}

For cotangent bundles, we have the following result.

\begin{corollary} \label{corollary/tangent}
If $E = T^{*}N$ is the cotangent bundle of a closed orientable manifold $N$ and the Euler characteristic $\chi(N) \ne \pm 1$, 
then the topology of the level sets for $H = K + V\circ \pi \colon T^{*}N \to {\mathbb R}$ always 
changes when $H$ passes a simple critical level.
This is the case, in particular, if 
\begin{enumerate}
\item the dimension $\dim(N)$ is odd or $\dim(N) = 2k,$ where $k$ is odd;
\item the Betti number $b_{\dim(N)/2}(N)$ is even.
\end{enumerate}
\end{corollary}

\begin{remark}
\mbox{ }

\begin{enumerate}

\item 
We remark that Theorem~\ref{theorem/closed_manifold} and Corollary~\ref{corollary/one_global_max} do not hold if one does not make 
any assumptions on $N$ or $E$. 
For instance, consider the tautological line bundle over $N = \mathbb C \mathbb P^1$. 
Then, for any smooth function $V$ on
$\mathbb C {\mathbb P}^1$ (with a unique non-degenerate maximum), we have that  
$$H^{-1}(h_{\max} + \varepsilon) \cong H^{-1}(h_{\max} - \varepsilon)  \cong S^3,$$
since $\pi \colon  H^{-1}(h_{\max} + \varepsilon)\to \mathbb C \mathbb P^1$ is 
isomorphic to the Hopf bundle.

\item

By considering rank $4$ vector bundles $\pi \colon E\to S^4$, 
one can even have a situation when $H^{-1}(h_{\max} + \varepsilon)$ 
is homeomorphic to $H^{-1}(h_{\max} - \varepsilon) \cong S^7$,
but not diffeomorphic to it. 

\item 
For cotangent bundles, the situation is a bit different.\\ 
First, we note that 
Corollary~\ref{corollary/tangent} applies to all $2$-dimensional orientable surfaces 
(in this case, the Euler characteristic of $N$ is even),  parallelizable manifolds $N$ 
(in particular, to all Lie groups), and odd-dimensional manifolds.\\  
Thus, there are no counterexamples in dimensions $n = 3$ or less.\\ 
We conjecture that the $4$-manifold $\mathbb C \mathbb P^2$\#${\mathbb T}^4$  
with Euler characteristic 
\[\chi(\mathbb C \mathbb P^2\mbox{\#}{\mathbb T}^4) 
=\chi(\mathbb C \mathbb P^2) + \chi(\mathbb T^4)-\chi(S^4)=3-2=1\] 
is a counterexample.\\
We note that in the case of non-orientable $N$, we have 
$\chi(\mathbb RP^2) = 1$, but 
$H^{-1}(h_{\max} - \varepsilon)$ and $H^{-1}(h_{\max} + \varepsilon)$
are not diffeomorphic; they are diffeomorphic to $S^2 \times S^1$ and the 
lens space $L(4,1)$, respectively~\cite{Ko}. 

\end{enumerate}

\end{remark}

In fact, for a class of bundles, including bundles over spheres, 
we have a much stronger statement, which follows from Adams' result \cite{Ada}. 

\begin{proposition} \label{proposition}
Consider a rank $n$ vector bundle $\pi \colon E \to N$ and a smooth function on $E$ of the form
$$
H = K + V
\circ\pi.
$$
If the restriction $\pi \colon L \to N \setminus U$ is a trivial bundle and $n \ne 2, 4$ or $8,$ 
then the topology of $H$ level sets changes when passing a simple critical level.
\end{proposition}

\begin{proof}
The statement follows from Adams' result ($S^{n-1}$ is an $H$-space only in dimensions $n = 1,2,4$ and $8$) \cite{Ada}.
\end{proof}

\begin{corollary}
If $N$ is a homotopy $n$-sphere, a situation of no topology change is possible only when 
$n = 2, 4$ or $8$.
\end{corollary}

\section{Applications} \label{section/nbody_problem}

\subsection{Quadratic spherical pendulum} \label{subsection/qsp}

Let $S^2$ denote the unit sphere in $\mathbb R^3(x,y,z)$. Consider the Hamiltonian system on $T^*S^2$ given by the energy function
$$H = \frac{1}{2}\langle p, p \rangle + V(z),$$ 
where $V = z^2 -z/2$ is the potential. This Hamiltonian system is integrable and is called a \textit{quadratic spherical pendulum}; it naturally appears in the context of
integrable Hamiltonian systems with non-trivial monodromy \cite{Ef}; see also \cite{BF} for the necessary background. The corresponding bifurcation diagram, that is, the set of the critical values
of the energy-momentum map $(H,J) \colon T^*S^2 \to \mathbb R^2$, is depicted in Fig.~\ref{fig/qsp}; here $J$ is the angular momentum about the $z$-axis --- the 
first integral of the system.

\begin{figure}[htbp]
\begin{center}
\includegraphics[width=90mm]{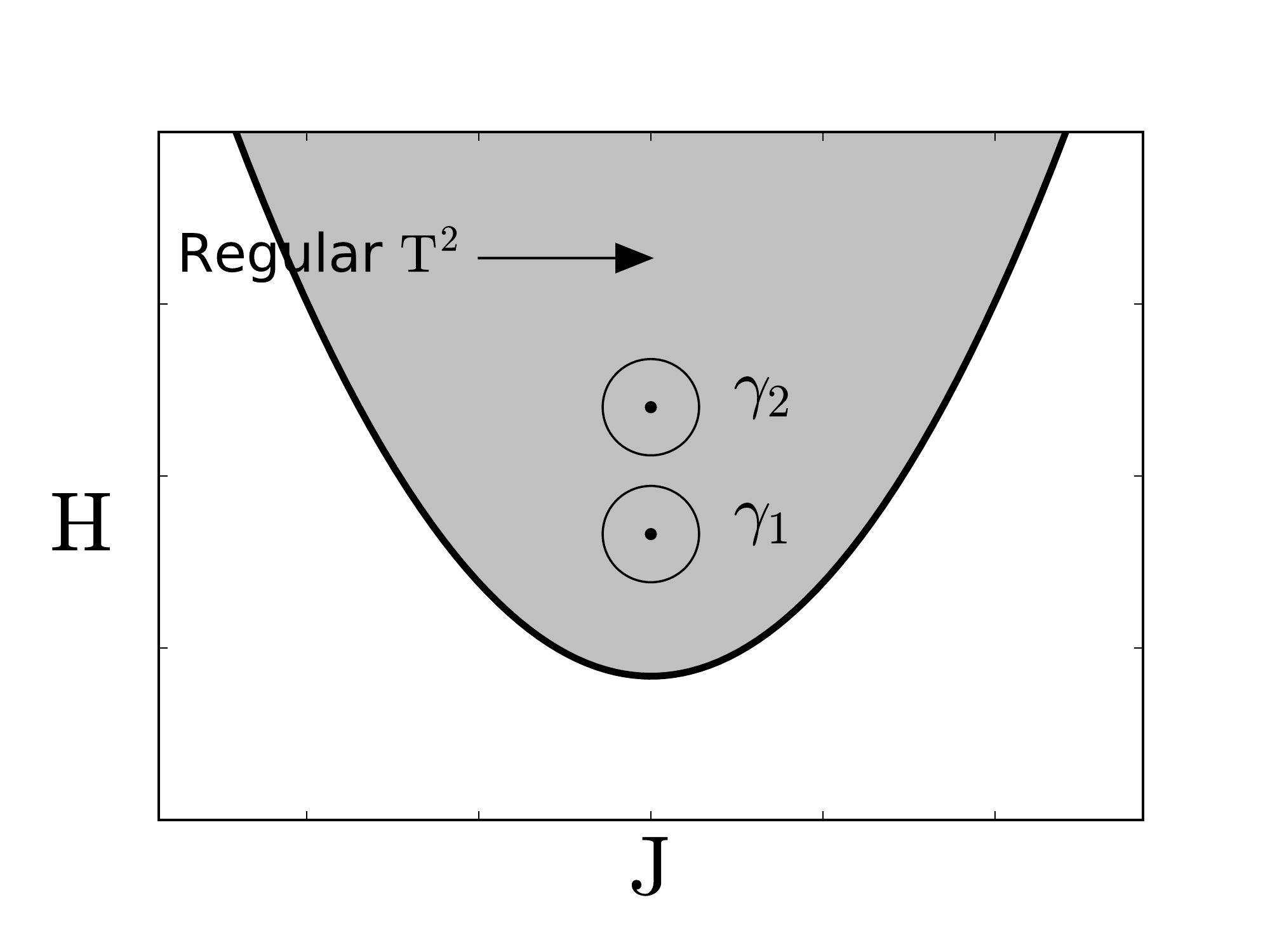}
\caption{Bifurcation diagram of the energy-momentum map $(H,J)$, and the two curves $\gamma_i$ corresponding to 
non-trivial monodromy.}
\label{fig/qsp}
\end{center}
\end{figure}

The Hamiltonian function $H$ has two non-degenerate critical points, which correspond to the two maxima of the potential $V|_{S^2}.$
These points gives rise to two critical level sets. From Theorem~\ref{theorem/closed_manifold}, we conclude that the topology of $H^{-1}(h)$ changes
when passing each of these critical levels. The same is true when we compare the topology of $H^{-1}(h)$ below and above these two critical levels. 
Indeed, the Euler characteristic $\chi(S^2) = 2$ is different from $0$ and $\pm 1$. In fact, it can be shown that these level sets are diffeomorphic to $S^2\times S^1, S^3$, 
and $\mathbb RP^3.$ We note that a similar  result applies also to the usual spherical pendulum, that is, when the potential $V(z) = z$.

One important consequence of the topology change for such system is the non-triviality of monodromy around the corresponding singular points; see \cite{Du, MBE}
for more details and the background.

\subsection{Restricted three-body problem}

The planar circular restricted $3$-body problem
with mass ratio $\mu\in(0,1)$ can be written as an autonomous 
Hamiltonian system on
$\mathbb R^2 \times \big(\mathbb R^2 \setminus \big\{(-\mu, 0),(0,1-\mu)\big\}\big)$ 
in a co-rotating reference frame; the Hamiltonian function is given by
$$
H = \dfrac{x'^2+y'^2}{2} - \frac{1}{2}(x^2+y^2) - \left(\dfrac{1-\mu}{r_1} + \dfrac{\mu}{r_2}\right),
$$
$r_i$ denoting the distances to the respective centers.
There are five equilibrium points $L_1, \ldots, L_5$. 
These are the critical points of the potential function
$$
V = - \frac{1}{2}(x^2+y^2) - \left(\dfrac{1-\mu}{r_1} + \dfrac{\mu}{r_2}\right).
$$
Each of this critical points gives rise to a bifurcation value for the energy function $H$. 
It is known that each such 
value gives rise to a topology change of the energy levels $H^{-1}(h)$, and 
it is not difficult to determine the homotopy types of these energy levels; the corresponding Hill regions are shown in Fig.~\ref{fig/Hill}. 
Below we show how this result of the topology change follows 
from the theory developed in this paper.

\begin{figure}[htbp]
\begin{center}
\hspace{-3mm} \includegraphics[width=128mm]{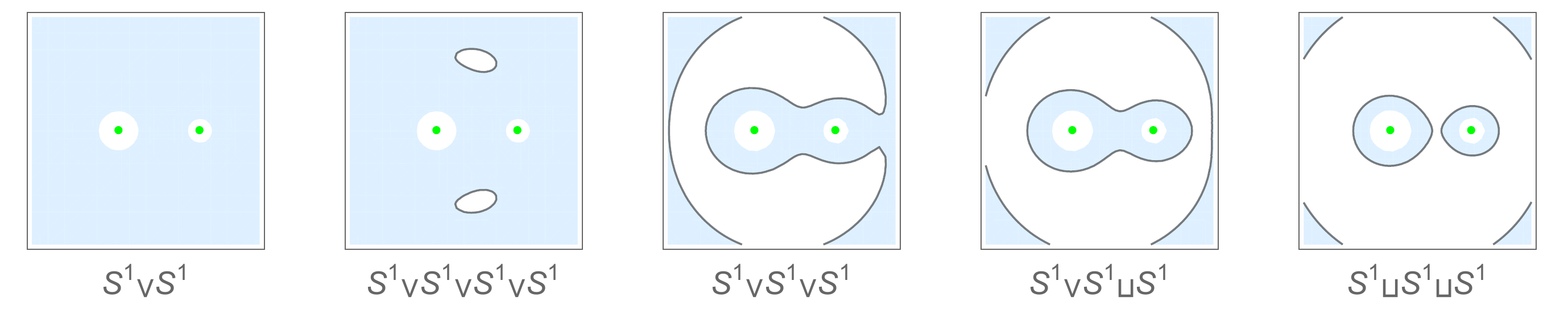}
\caption{The Hill regions and their homotopy types.}
\label{fig/Hill}
\end{center}
\end{figure}

First, we note that $L_1, L_2$ and $L_3$ are index $1$ critical points for $V$ and hence also for $H$. Hence the topology  changes
when passing these critical values by Theorem~\ref{theorem:level}. 
The critical points $L_4$ and $L_5$ are of index $2$, which is half the dimension of the 
phase space, and are on the same energy level; these points are related
by the $\mathbb Z_2 = O(2)/SO(2)$ symmetry of the problem. 
Since the base manifold is not compact and there are two critical points, the result follows.

\subsection{Planar $n$-body problem} 

Consider the Newtonian $n$-body problem in $\mathbb R^2$. The Hamiltonian of this problem is given by the function
$$
H = \sum_{i=1}^n \dfrac{\|p_i\|^2}{2m_i} - \sum_{i<j}\dfrac{Gm_im_j}{\|q_j-q_i\|},
$$
where $G$ is the gravitational constant. Reducing by the translational symmetry, 
we get a Hamiltonian system on $T^*Q$, where $Q = \mathbb R^{2n-2} \setminus \triangle$ with $\triangle$ denoting
the (reduced) collision set, that is, the set of points where $q_i = q_j$ for some $i \ne j$, reduced by translations.

Fixing a non-zero value of the angular momentum $L = \sum_i q^x_ip^y_i-q^y_ip^x_i$ and taking the quotient with respect to
the $SO(2)$ symmetry group, one gets the reduced symplectic manifold $M$ of dimension $4n-6$. The Hamiltonian $H$ restricts
to this manifold as a smooth function. Following Smale \cite{Sm1, Sm2}, we are interested in the topology of the level sets 
of $H$ on this reduced manifold. Specifically, we would like to answer the general question of whether the topology of $H^{-1}(h)|_M$ always changes
when passing a bifurcation level.

We observe that the manifold $M$ can be viewed as a vector bundle over $\mathbb R_{+} \times {\mathbb C}{\rm P}(n-2)$. Here $\mathbb R_{+} = (0, \infty)$
and each point of $\mathbb R_{+}$ corresponds to fixing the moment of inertia
$$
I = \frac{1}{2}\sum_i m_i\|q_i\|^2.
$$
The reduction of the Hamiltonian to $M$ can be rewritten (\cite{Ala}) in the following form:
$$
H = K + c^2/(4\rho^2) - U(q)/\rho,
$$
where $K$ is a bundle metric, $c \ne 0$ is the value of $L$, $I = \rho$ and $U$ denotes the reduction of the potential to the projective space $I^{-1}(1)/\mathbb S^1.$ 
From this description, it follows that the index $\lambda$ of each non-degenerate critical point is at most $2n-4< \frac{1}{2}(4n-6) = 2n-3.$ 
From Theorem~\ref{theorem:level}, we get the following result, which also follows from McCord \cite[Proposition 5.2]{McC}.

\begin{theorem} \label{theorem/planar}
Consider  the Hamiltonian  $H$ on the reduced manifold $M$. 
Assume that $V$ defines a Morse function on 
$I^{-1}(1)/\mathbb S^1 = {\mathbb C}{\rm P}(n-2) \setminus \tilde{\Delta}$, so that
$H$ is a Morse function as well. 
Assume, moreover, that a given critical level set $H^{-1}(h_c)$ contains a single 
critical point or two  critical points of the same index 
(related by the $\mathbb Z_2 = O(2)/SO(2)$ symmetry). 
Then the topology of $H^{-1}(h)$ changes when passing this critical level.
\end{theorem}

\section{Acknowledgements}

We are very grateful to A.\ Albouy for useful discussions and his comments, which led to improvement of the original version of the article. 
The authors would also like to thank participants of A.T. Fomenko's seminar 
{\em Modern Geometry Methods} for useful comments.

\section{Appendix} \label{section/appendix}

The goal of this section is to prove Theorem~\ref{theorem/closed_manifold} formulated in Section~\ref{sec:natural}. Recall that we consider a smooth function of the form
\begin{equation*}
H = K + V \circ \pi
\end{equation*}
on a rank $n$ vector bundle $\pi \colon E \to N$.
Here $N$ is a connected $n$-manifold without boundary, $K$ is a Riemannian bundle metric on $E$ and $V$ is a Morse function on $N$. 
The manifold $N$ and the bundle $\pi \colon E \to N$ are assumed to be orientable. 
In this section, the manifold $N$ is assumed to be compact.

Observe that a section $s \colon N \to E$ in general position has finitely many zeros and that, 
by homogeneity, they can be assumed to be arbitrary close to some point $x \in N$. 
This shows that $E$ admits an almost global section, which is defined and non-zero 
everywhere on $N \setminus D$ (and also on $N  \setminus \{x\}$),
where $D$ is an arbitrary small disk containing $x$. 

We observe that for the unit sphere bundle $\tilde{\pi} =\pi|_{S_1N}$, the degree of the map
$$
r \circ s \colon  \partial D \cong S^{n-1} \to S^{n-1} \cong \tilde{\pi}^{-1}(x),
$$
where $r$ is a retraction of $\tilde{\pi}^{-1}(D)$ onto the central fiber $\tilde{\pi}^{-1}(x),$ is equal 
to the intersection number of $s(N)$ and the zero section $N \subset E$.  

We will need the following lemma.

\begin{lemma} \textup{(\cite[\S 19.6B]{FF},\cite[Section 4.D]{Ha})}\quad\\
The Euler number $e(E)$, that is, the pairing of the Euler class of $E$ with $N$, is given by 
the intersection number of the zero section $N \subset E$ and 
a section in general position.
If $E = T^{*}N$ is the cotangent bundle, the Euler number equals the Euler characteristic of $N$. 
\end{lemma}

We will also need the following result.

\begin{lemma} \label{lemma/section_modified}
Let $\pi \colon E \to B$ be an orientable $n$-vector bundle over a manifold $B$ 
(possibly with boundary). Assume that there exists a global, everywhere non-zero section $s$
of $\pi$. Let $\tilde{\pi} =\pi|_{S_1B} \colon S_1B \to B$ be the unit sphere bundle of $\pi$ 
(with respect to some bundle metric). Then, for any coefficient group $G$,
the relative homology groups of $(S_1B, \partial S_1B)$ are the same as for the direct product $(B \times S^{n-1}, \partial B \times S^{n-1})$. 
Moreover, for all $b\in B$ and any $G$,
$\tilde{\pi}^{-1}(b)$ represents a non-trivial homology class in $H_{n-1}(S_1B, G)$.
\end{lemma}

\begin{proof}
Take any simplicial decomposition $K$ of $B$ and the standard cellular decomposition 
$\{pt\} \cup D^{n-1}$ of $S^{n-1}$. 
Then construct a cellular decomposition of $S_1B$ as follows. 
For any simplex $c_k \in K$, consider the preimage $\pi^{-1}(c_k)$. It is a direct product 
$c_k \times \mathbb R^n$.
Without loss of generality, the section $s$ has the form 
$b \mapsto e_1 = (1,0, \ldots,0) \in S_b^{n-1} \subset \mathbb R^n$,
where $S_b^{n-1} = \tilde{\pi}^{-1}(b)$ is the fiber of $\tilde{\pi}$ over $b \in c_k$. 
The preimage  $\tilde{\pi}^{-1}(c_k)$ is thus 
a direct product $c_k \times S^{n-1}$. Moreover, it admits a cellular decomposition of the form 
$$c_k \times \{e_1\} \cup c_k \times D^{n-1}.$$
Since for any $k$-cell $c_k \in K$, the distinguished point $e_1$ is given by the section $s$, 
we have that the boundary operator for $S_1B$ satisfies
\begin{align*}
\partial (c_k \times \{e_1\}) &= (\partial c_k) \times \{e_1\}.
\end{align*}
Moreover, since the bundle $\pi$ and hence $\tilde{\pi}$ are trivial 
over the closure $\overline{c_k},$ we have that 
\begin{align*}
\partial (c_k \times \{D_{n-1}\}) &= \sum_{i} \pm (-1)^i c_k^i \times \{D_{n-1}\},
\end{align*}
where $\partial c_k = \sum (-1)^i c_k^i$ is the boundary of $c_k$ with the induced orientation. 
We observe that 
since $\pi$ and hence $\tilde{\pi}$ are orientable, the sign can be chosen so that 
$\partial (c_k \times \{D_{n-1}\}) = \partial c_k \times \{D_{n-1}\}.$ 
We conclude that the boundary operator is the same as for the direct product. 

To prove the last statement, consider the cell $\{b_0\} \times D^{n-1}$, 
where $b_0$ is a vertex of $K$. We observe that this cell 
is not a boundary of $C_1 \times D^{n-1}$ or $C_{n} \times \{e_1\}$ (here 
$C_1$ and $C_n$ are some $1$ and $n$ chains in $K$, respectively). 
Indeed, $\partial C_1$ consists of an even number of points and the boundary
$\partial(C_{n} \times \{e_1\})$ is transverse to the fibers.
\end{proof}

\begin{remark}
From Lemma~\ref{lemma/section_modified} it follows that if $\pi|_{S_1B} \colon S_1B \to B$ admits a global section, then the homology
groups of ${S_1B}$ can be computed using a K\"unneth formula.
\end{remark}

\begin{example}
 As an example, consider the Stiefel manifold $V_{k,2}$. It can be viewed as the unit tangent bundle
 of $S^{k-1}$. It is known that the 
 integer homology groups of $V_{2k,2}$ are the same as for the product $S^{2k-1}\times S^{2k-2}$; see \cite[Section 3.D]{Ha}.  On the other hand,
 $V_{2k,2}$ is not homeomorphic to  the product $S^{2k-1}\times S^{2k-2}$, unless $k = 1, 2$ or $4$ \cite{JW, Ada}. We note that the integer homology groups 
 of $V_{2k+1,2}$ are different from the homology groups of $S^{2k}\times S^{2k-1}$; in this case there is no global section since the base $S^{2k}$ is even-dimensional.
\end{example}

Let $e(E)$ denote the Euler number of $\pi \colon E \to  N$. We are ready to prove the desired result (Theorem~\ref{theorem/closed_manifold}).

\begin{theorem} \label{theorem/closed_manifold_appendix}
  Let $N$ be a closed orientable $n$-manifold and $\pi \colon E \to N$ be an orientable rank $n$ vector bundle over $N$. Consider 
  the function $H = K + V\circ \pi \colon E \to {\mathbb R}$ on $E$
  and let $x_1, \ldots, x_L \in V^{-1}(h_c)$ be the non-degenerate global maxima of the function $V$ on $N$.  Then
  the topology of $H$ changes when $H$ passes the critical value $h_c$ if one of the following conditions is satisfied
  
  $1)$ $L = 1$ and the Euler number $e(E)$ is not equal to $\pm 1$;
 
 $2)$ $L = 2$ and the Euler number $e(E)$ does not vanish;

$3)$ $L > 2.$
\end{theorem}

\begin{proof}
 The case $L > 2$ was considered earlier, so we only need to consider the cases $L = 1$ and $L = 2$.
 
$\bullet$ \textit{Case $L = 1$.}
 
 Let $x_{\max}$ be the unique non-degenerate global maximum of $V$ and $h_{\max} = H(x_{\max},0) = V(x_{\max}).$
Fix a small number $\varepsilon > 0$. 
Then the level sets $\Sigma_{\pm} := H^{-1}(h_{\max} \pm \varepsilon)$ are regular. We observe that
the level $\Sigma_{+}$ is homeomorphic to the unit sphere bundle $S_1N$ of $E$. 

Let $T_{\pm} = H^{-1}(h_c\pm\varepsilon) \setminus \pi^{-1}(U)$, where $U$ is a small open disk containing the maximum $x_{\max}$. Observe 
that the sets $T_{-}$ and $T_{+}$ are  sphere bundles over $L \setminus U$, and that these sphere bundles are isomorphic through a radial projection.
Similarly to the proof of Proposition~\ref{proposition/maxima}, we have that
\begin{align*}
\Sigma_{-} &\cong  T_{-} \cup_{f_-} S^{n-1} \times D^n \ \mbox{ and } \\ 
\Sigma_{+} &\cong  T_{+} \cup_{f_+} D^n \times S^{n-1},
\end{align*}
where the attaching maps for the respective boundary components 
are such that $f_{-}(\{x\} \times \partial D^n)$ is a fiber of the sphere bundle
$\pi \colon T_{-} \to N \setminus{U}$ and $f_{+}$ comes from the sphere bundle structure on $\Sigma_{+}$. In particular, $f_{+}(\{x\} \times S^{n-1})$ is a fiber
of $\pi \colon T_{+} \to N \setminus{U}$ and
$f_{+}(\partial D^n \times \{y\})$ is a section over $\partial U$.

Consider a part of the homology exact sequences of the pairs
$(\Sigma_\pm,T_\pm)$:
\begin{equation} \label{eq/exactseq}
G \stackrel{\partial_{*}}{ \longrightarrow} H_{n-1}(T_\pm, G) \longrightarrow
H_{n-1} \big(\Sigma_\pm, G \big)  \longrightarrow 0.
\end{equation}
Observe that the map $H_{n-1}(T_{-}, G) \to H_{n-1} \big(\Sigma_-, G \big)$ 
has a non-trivial kernel,
given by the homology class of the fiber $f_{-}(\{x\} \times \partial D^n)$; this
homology class is non-trivial in $H_{n-1}(T_{-}, G)$ by Lemma~\ref{lemma/section_modified}. 
On the other hand, this map is surjective.

The bundle $\pi \colon T_{+} \to N \setminus{U}$ also has a global section. We denote it by $f$. We observe that
the restriction of $f_{+}^{-1} \circ f$ to the boundary sphere $\partial U \cong S^{n-1}$ is a map of degree
$k\ne \pm 1$, where $k$ is the Euler number. (Strictly speaking,  $f_{+}^{-1} \circ f$ maps into $D^n \times S^{n-1}$, but this space deformation retracts onto
 $S^{n-1}$.) The section $\partial D^n \times \{y\}$ gives rise to a map of degree $0$.

Setting $G = \mathbb Z_k$ for $k \ne 0$ and $G = \mathbb R$ for $k = 0$,
we get that $f(\partial U)$ and $f_{+}(\partial D^n \times \{y\})$ are of the same $G$-homology class 
in $\partial T_{+}$. 
Hence $[f_{+}(\partial D^n \times \{y\})]$ is trivial in the group $H_{n-1}(T_{+},G)$ and 
$$
H_{n-1} \big(\Sigma_+, {G}\big) \cong  H_{n-1}(T_{+},G)
$$
when $G = \mathbb Z_k$ (or $\mathbb R$ when $k = 0$). Since $T_{-}$ and $T_{+}$ are homeomorphic and also compact,
it follows that the $(n-1)$-th homology groups $H_{n-1}(\Sigma_+, G)$ and
$H_{n-1}(\Sigma_-, G)$ are not isomorphic for $G = \mathbb Z_k$. 

$\bullet$ \textit{Case $L = 2$.}
 
We shall assume that the two maxima are on close, but 
 different level sets of $V$,
 and that we are passing both of these maxima at the same time. Then one maximum becomes local and the other  global.
 
Consider what happens when we pass the second (global) maximum. 
 We observe that there is an almost 
 global section $f$ such that the restriction of $f_{+}^{-1} \circ f$ to the boundary sphere $\partial U \cong S^{n-1}$ is a map of degree
$k\ne \pm 1$, where $k$ is the Euler number. By the assumption, $k \ne 0$. 
Observe that $f(\partial U)$ is trivial in the group $H_{n-1}(T_{+},G)$. 
However, if $G = \mathbb R$, then  $f_{+}(\partial D^n \times \{y\})$
is non-trivial in $H_{n-1}(T_{+},G)$. Indeed, using a suitable cellular decomposition of $T_{+}$ (see Lemma~\ref{lemma/section_modified}),
we get that any relative $n$-cycle in $T_{+}$ is given by linear combinations of

a) The products of relative $1$-cycles in $(N\setminus U,\partial U)$ and $D^{n-1}$, where $D^{n-1} \cup\{pt\} = S^{n-1}$;

b) The section $f(N\setminus U)$. 

But the boundary of any $n$-cycle as in a) vanishes in $\partial U \times S^{n-1}$, whereas 
$\partial f(N\setminus U)$ and $f_{+}(\partial D^n \times \{y\})$ are different homology cycles in $\partial U \times S^{n-1}$ since $k \ne 0$. 
It follows that $f_{+}(\partial D^n \times \{y\})$ is not a boundary and that 
$$
b_{n-1}(\Sigma_+, \mathbb R) +1 =   b_{n-1}(T_{+},\mathbb R).
$$
Since 
$$
b_{n-1}(\Sigma_-, \mathbb R) +1 =   b_{n-1}(T_{-},\mathbb R)
$$
and since $T_{-}$ and $T_{+}$ are homeomorphic and also compact,
we get that $b_{n-1}(\Sigma_+, \mathbb R) = b_{n-1}(\Sigma_-, \mathbb R).$
But Proposition~\ref{proposition/maxima}  implies that the other (local) maximum contributes to the change of the $(n-1)$-Betti number,
also when $G = \mathbb R$. The result follows.
\end{proof}

\begin{remark}
Assume that the maxima $x_1, \ldots, x_L$ are not located on one critical level, but belong 
 to (the interior of) the set $V^{-1}[a,b], \ a < b,$ that contains no other critical points. In this case, the same result holds if one compares the topology
 of $H^{-1}(a)$ with that of $H^{-1}(b)$; cf. Subsection~\ref{subsection/qsp}.
\end{remark}

\begin{remark}
 We note that if $L = 1$ and the Euler class vanishes, then the $(n-1)$-Betti number changes according to 
 $$
b_{n-1}(H^{-1}(h_c+\varepsilon),\mathbb F) = b_{n-1}(H^{-1}(h_c-\varepsilon),\mathbb F) + 1,
$$
where $\mathbb F$ is a field; cf. Proposition~\ref{proposition/maxima}. 
\end{remark}

\end{document}